\documentclass[11pt]{article}
\usepackage{amsmath,amssymb,amsthm}
\usepackage[margin=1.1in]{geometry} 
\usepackage[colorlinks=true,urlcolor=blue]{hyperref}

\newtheorem{theorem}{Theorem}
\newtheorem{proposition}[theorem]{Proposition}
\newtheorem{lemma}[theorem]{Lemma}
\theoremstyle{definition}
\newtheorem{definition}[theorem]{Definition}
\newtheorem{corollary}[theorem]{Corollary}
\newtheorem{conjecture}[theorem]{Conjecture}
\theoremstyle{definition}
\newtheorem{example}[theorem]{Example}
\newtheorem{remark}[theorem]{Remark}

\newcommand{\Z}{\mathbb{Z}}
\newcommand{\F}{\mathbb{F}_2}
\newcommand{\ind}[1]{[\,#1\,]}
\DeclareMathOperator{\ord}{ord}

\title{Three Quadratic Recurrences Modulo 3:\\ What We Can Prove and What We Can Guess}
\author{Steven Finch\\
MIT Sloan School of Management}
\date{\today}

\begin{document}
\maketitle

\begin{abstract}
We study three innocent-looking recurrences modulo $3$, each built from
squares of two earlier terms. One of them turns out to be secretly
\emph{linear}, and we can compute its periods exactly. One is a
\emph{permutation} hiding a small, perfectly orderly piece inside a large,
random-looking one. The third looks thoroughly random, but it still has a few
provable features. Along the way we meet open problems that need nothing
beyond a first course in abstract algebra to state.  This paper is a starting 
point, not the final word in any sense.
\end{abstract}

\section{The setup}

Fix an integer $k\ge 1$ and start a sequence $x_1,x_2,\dots$ of elements of
$\Z/3=\{0,1,2\}$ with
\[
x_1=x_2=\dots=x_{k-1}=0,\qquad x_k=1 .
\]
We continue it with one of three rules (all arithmetic mod $3$):
\begin{align*}
\text{(A)}\quad & x_n = x_{n-1}+x_{n-k}^2,\\
\text{(B)}\quad & x_n = x_{n-1}^2+x_{n-k},\\
\text{(C)}\quad & x_n = 2\bigl(x_{n-1}^2+x_{n-1}+x_{n-k}^2+x_{n-k}\bigr).
\end{align*}
Each new term depends on the previous $k$ terms, and there are only $3^k$
possible blocks of $k$ terms. So every sequence is \emph{eventually
periodic}: there are a least period $N\ge 1$ and a least pre-period
$\mu\ge 0$ such that $x_{n+N}=x_n$ for all $n>\mu$. We write $N_A(k)$,
$N_B(k)$, $N_C(k)$ (and similarly $\mu$) for the three rules.

\paragraph{The one trick we need.}
Write $\ind{P}$ for $1$ if the statement $P$ is true and $0$ if it is false.

\begin{lemma}\label{lem:squares}
For $c\in\Z/3$ we have $c^2=\ind{c\neq 0}$ and $2(c^2+c)=\ind{c=1}$.
\end{lemma}

\begin{proof}
Check the three cases:
\[
\begin{array}{c|ccc}
c & 0 & 1 & 2\\ \hline
c^2 & 0 & 1 & 1\\
2(c^2+c) & 0 & 1 & 0
\end{array}
\qedhere
\]
\end{proof}

So the three rules really read
\begin{align*}
\text{(A)}\quad & x_n = x_{n-1}+\ind{x_{n-k}\neq 0},\\
\text{(B)}\quad & x_n = x_{n-k}+\ind{x_{n-1}\neq 0},\\
\text{(C)}\quad & x_n = \ind{x_{n-1}=1}+\ind{x_{n-k}=1}.
\end{align*}
Rules (A) and (B) use the same two ingredients with the lags swapped. As we
will see, they behave completely differently.

\section{Recurrence (C): secretly linear}

Some data:
\[
\begin{array}{c|ccccccccccc}
k & 2&3&4&5&6&7&8&9&10&11&12\\ \hline
N_C(k) & 3&7&15&21&63&127&63&73&889&1533&3255
\end{array}
\]
and $\mu_C(k)=1$ for every $k\ge 2$. Many of these numbers are of the form $2^j-1$,
which hints at the field $\F=\{0,1\}$.

For a polynomial $f\in\F[x]$ with $f(0)=1$, its \emph{order} $\ord f$ is the
least $m\ge 1$ with $f(x)$ dividing $x^m-1$ (equivalently, $x^m\equiv 1 \pmod f$;
such an $m$ exists because $x$ is a unit modulo $f$). A polynomial of degree
$k$ is \emph{primitive} if its order is $2^k-1$; for both notions see
\cite[Ch.~3, \S1]{LN}.

\begin{theorem}\label{thm:C}
For $k\ge 2$ we have $\mu_C(k)=1$ and $N_C(k)=\ord(x^k+x+1)$.
In particular $N_C(k)\le 2^k-1$, with equality exactly when $x^k+x+1$ is
primitive.
\end{theorem}

\begin{proof}
\emph{Step 1: reduce to $\F$.}
Put $y_n=\ind{x_n=1}$. By the rule, $x_n=1$ exactly when exactly one of
$y_{n-1},y_{n-k}$ equals $1$. Hence, for $n\ge k+1$,
\[
y_n = y_{n-1}+y_{n-k}\quad\text{in }\F ,
\]
a \emph{linear} recurrence, with $y_1=\dots=y_{k-1}=0$ and $y_k=1$.

\emph{Step 2: the period of $y$.}
Let $v_n=(y_{n-k+1},\dots,y_n)\in\F^k$. Then $v_{n+1}=Av_n$, where $A$ is
the companion matrix of $p(t)=t^k+t^{k-1}+1$. Since $p(0)=1$, $A$ is invertible, 
so $y$ is purely periodic and can be implemented backwards. The vectors
$v_k,v_{k+1},\dots,v_{2k-1}$ form a basis: they are ``triangular'', because
the first $1$ of $y$ moves one place left each time. If $A^mv_k=v_k$, then
$A^m$ fixes every $A^iv_k$, so $A^m=I$. The period of $y$ is therefore the
order of the matrix $A$. That is the least $m$ with $p\mid t^m-1$, because
$p$ is the minimal polynomial of its companion matrix. Finally, $p$ and its
reversal $x^k+x+1$ divide exactly the same polynomials $x^m-1$, so they have
the same order. (This is the standard link between linear recurring sequences 
and orders of polynomials; see \cite[Ch.~8]{LN}.) The orbit of the nonzero 
vector $v_k$ never reaches $0$, so the period is at most $2^k-1$.

\emph{Step 3: back to $x$, and $\mu=1$.}
Implementing $y_n=y_{n-1}+y_{n-k}$ backwards from the initial data gives $y_0=1$,
$y_{1-k}=1$, and $y_{2-k}=\dots=y_{-1}=0$. A direct check shows
\[
x_n=y_{n-1}+y_{n-k}\quad(\text{as integers})\qquad\text{for all } n\ge 2 .
\]
For $n\ge k+1$ this is the rule itself, and for $2\le n\le k$ it is a quick
check. The right-hand side is purely periodic with the same period as $y$
(since $y_n=\ind{x_n=1}$ recovers $y$). At $n=1$ it equals
$y_0+y_{1-k}=2\neq 0=x_1$. So the sequence is periodic from $n=2$ on but not
from $n=1$, and $\mu_C(k)=1$.
\end{proof}

\begin{remark}
	The polynomial $x^k+x+1$ has no repeated factors, so $N_C(k)$ is the
	least common multiple of the orders of its irreducible factors
	\cite[Lemmas~3.1--3.2]{B}. For instance $x^5+x+1=(x^2+x+1)(x^3+x^2+1)$ gives
	$N_C(5)=\operatorname{lcm}(3,7)=21$.
\end{remark}

Parts (2) and (3) below are the easy halves of Bartholdi's
Propositions~2.2--2.3 \cite{B}.

\begin{corollary}\label{cor:C}
Let $q=2^j$.
\begin{enumerate}
\item If $k\equiv 2\pmod 3$, then $3\mid N_C(k)$.
\item $N_C(q)$ divides $q^2-1$.
\item $N_C(q+1)$ divides $q^2+q+1$.
\end{enumerate}
\end{corollary}

\begin{proof}
(1) For $k\equiv 2\pmod 3$ we have $x^k+x+1\equiv x^2+x+1\pmod{x^2+x+1}$,
because $x^3\equiv1$ modulo $x^2+x+1$. Hence $x^2+x+1$ divides $x^k+x+1$. A
divisor's order divides the order, and $\ord(x^2+x+1)=3$.

For (2) and (3) we use the ``freshman's dream'' $(a+b)^2=a^2+b^2$, valid in
characteristic $2$, which gives $(a+b)^q=a^q+b^q$.

(2) Modulo $f=x^q+x+1$ we have $x^q\equiv x+1$. So
$x^{q^2}\equiv(x+1)^q=x^q+1\equiv x$, and $x^{q^2-1}\equiv 1$ because $x$ is
a unit.

(3) Modulo $f=x^{q+1}+x+1$ we have $x^{q+1}\equiv x+1$. So
$x^{q^2+q}\equiv(x+1)^q=x^q+1$, and therefore
$x^{q^2+q+1}\equiv x^{q+1}+x\equiv 1$.
\end{proof}

In fact equality holds in (2) and (3): $N_C(2^j)=4^j-1$ for $j\ge1$ and
$N_C(2^j+1)=4^j+2^j+1$ for $j\ge 0$, as proposed in \cite{IMO}; see the 
proof in \cite{B} and a list of periods in OEIS A046932 \cite{OEIS}.

\begin{conjecture}
$N_C(2^j-1)=2^{N_C(j)}-1$ for $j\ge 2$.
\end{conjecture}

Bartholdi verified this for $j\le 16$ by computer; for instance
$N_C(3)=2^3-1$, $N_C(7)=2^7-1$, $N_C(15)=2^{15}-1$, and $N_C(31)=2^{21}-1$
(note $N_C(5)=21$). He also proved half of the 
conjecture \cite[Prop.~2.4]{B}; here is a short proof in our language.

\begin{proposition}
For $j\ge2$, $N_C(2^j-1)$ divides $2^{N_C(j)}-1$.
\end{proposition}

\begin{proof}
Let $q=2^j$, $f=x^{q-1}+x+1$, and work in the ring $R=\F[x]/(f)$. The
squaring map $\sigma(r)=r^2$ is additive on $R$. For a polynomial
$h(s)=\sum c_is^i\in\F[s]$, define the map $h(\sigma)(r)=\sum c_i r^{2^i}$.
Because $\sigma$ is additive, $(gh)(\sigma)=g(\sigma)\circ h(\sigma)$. In $R$
we have $x^q=x\cdot x^{q-1}=x(x+1)=x^2+x$, that is, $g(\sigma)(x)=0$ for
$g(s)=s^j+s+1$. Let $M=N_C(j)=\ord g$ and write $s^M+1=h(s)g(s)$. Then
$x^{2^M}+x=h(\sigma)\bigl(g(\sigma)(x)\bigr)=0$. So $x^{2^M-1}=1$ in $R$, and
therefore $\ord f$ divides $2^M-1$.
\end{proof}

What is missing for the conjecture is to show that $x$ does not have a
\emph{smaller} order in $R$.

\section{Recurrence (B): a permutation with an orderly corner}

Some data:
\[
\begin{array}{c|ccccccccc}
k&2&3&4&5&6&7&8&9&10\\ \hline
N_B(k)&8&22&45&138&415&916&3998&13142&38763
\end{array}
\]
and $\mu_B(k)=0$ for all $k\ge2$. The reason is easy.

\begin{proposition}\label{prop:Bperm}
For $k\ge2$ the map $T(w_0,\dots,w_{k-1})=(w_1,\dots,w_{k-1},\,w_0+w_{k-1}^2)$
is a bijection of $(\Z/3)^k$. Hence every orbit is a cycle and
$\mu_B(k)=0$. Its only fixed point is $0$.
\end{proposition}

\begin{proof}
From $T(w)$ we can recover $w_0=(\text{last entry})-w_{k-1}^2$, so $T$ is
injective, and hence bijective on a finite set. A bijection of a finite set
moves every point around a cycle. If $T(w)=w$, then $w$ is constant, say
equal to $c$, with $c=c+c^2$, so $c=0$.
\end{proof}

Thus $N_B(k)$ is the length of the cycle of the permutation $T$ through
$(0,\dots,0,1)$. Computer decompositions for $k\le 20$ show two kinds of
cycles: thousands of short cycles of length $4k$, and a handful of huge ones.
The short ones can be described exactly.

\begin{theorem}\label{thm:B}
Let $k\ge2$. The number of states fixed by $T^{4k}$ is
\[
1+4\!\!\sum_{e\equiv 3-k \ (\mathrm{mod}\ 4)}\!\binom{k}{e}
\;=\;1+2^k+2^{k/2+1}\cos\frac{3\pi(k-2)}{4}.
\]
\end{theorem}

\begin{proof}
Since $T$ is a bijection, states fixed by $T^{4k}$ correspond to two-sided
sequences $(x_n)_{n\in\Z}$ of period dividing $4k$ satisfying
$x_n=x_{n-k}+b_{n-1}$, where $b_n=\ind{x_n\neq0}$. Arrange one period in a
$4\times k$ array, with $x_{j+tk}$ in row $t\in\Z/4$ and column
$j\in\{0,\dots,k-1\}$. Going down a column, the four increments
$x_{j+tk}-x_{j+(t-1)k}=b_{j+tk-1}$ are the $b$-values of the previous column
(for $j=0$, the last column shifted by one row). They add up to $0$ mod $3$.

\emph{Each column has $1$ or $4$ zeros.} The number of nonzero entries in a
column is $0$ or $3$, because it is a multiple of $3$ and at most $4$. If some
column is all zeros, the next column has all increments $0$, so it is
constant. It must then also be all zeros, since a constant nonzero column has
no zeros at all. Repeating this around the array gives $x\equiv0$, the fixed
point.

\emph{Where the single zeros sit.} Suppose the zero of one column is in row
$u$. The next column then reads $a,a,a+1,a+2$ in rows $u-1,u,u+1,u+2$. The
value $a=0$ would give two zeros, $a=1$ puts the zero in row $u+2$, and $a=2$
puts it in row $u+1$. In terms of positions mod $4k$, consecutive zeros (one
per column) differ by $k+1$ or $2k+1$, and the wrap-around from the last
column to the first behaves the same way. Conversely, any such choice of zero
positions determines the columns and gives a solution.

\emph{Counting.} Choose the zero in column $0$ ($4$ ways), then a sequence of
$k$ steps, each $k+1$ or $2k+1$, that returns to the start. If $e$ steps
equal $2k+1$, the total is $k(k+1)+ek$, and this is $\equiv0\pmod{4k}$
exactly when $e\equiv 3-k\pmod 4$. The closed form follows from the standard
roots-of-unity filter with $(1+i)^k$.
\end{proof}

For example, at $k=2$ every nonzero state is counted ($8$ of them), matching
$N_B(2)=8$. At $k=5$ the count is $40$, which is two $20$-cycles.

\begin{corollary}
For $k\ge3$, $N_B(k)$ does not divide $4k$.
\end{corollary}

\begin{proof}
In the theorem's solutions no two zeros are adjacent, because neither $k+1$
nor $2k+1$ is $\equiv\pm1\pmod{4k}$. The starting block $0,0,\dots,0,1$ has
adjacent zeros when $k\ge3$.
\end{proof}

So the orderly corner, about $2^k$ of the $3^k$ states, misses our starting
point. The remaining states resemble a random permutation: for
$k\le 20$ the largest cycle averages about $0.63\cdot 3^k$, close to the
Golomb-Dickman constant $0.6243\ldots$ predicted for random permutations 
\cite{F1,F2}.

\begin{conjecture}
As $k\to\infty$, the ratios $N_B(k)/3^k$ behave like independent uniform
random numbers in $(0,1)$.
\end{conjecture}

\section{Recurrence (A): looks random, but not entirely}

Some data:
\[
\begin{array}{c|ccccccccccc}
k&2&3&4&5&6&7&8&9&10&11&12\\ \hline
N_A(k)&4&4&9&19&4&4&22&36&4&4&45
\end{array}
\]
The value $4$ keeps reappearing. Here the pre-periods are huge: for example
$\mu_A(403)\approx1.0\times10^{11}$ and $\mu_A(487)\approx4.8\times10^{12}$.
Throughout this section let $b_n=\ind{x_n\neq0}$ and
$d_n=x_n-x_{n-1}=b_{n-k}$ for $n\ge k+1$.

\begin{proposition}\label{prop:A4}
For $k\ge2$ we have $N_A(k)\ge4$.
\end{proposition}

\begin{proof}
First, $x$ is not eventually $0$. Otherwise, let $M$ be least with $x_n=0$ for
all $n\ge M$. Then $M>k$, and for $n\ge M+1$ we get $b_{n-k}=d_n=0$. That
means $x_n=0$ for all $n\ge M+1-k<M$, contradicting the choice of $M$.

Now take a window of one full period $n=m+1,\dots,m+N$, with $m$ large. The
increments $d_n$ add up to $x_{m+N}-x_m=0$ mod $3$. But their sum is the
number $Z$ of nonzero terms in another full period, the one shifted by $k$.
So $3\mid Z$, and $Z>0$, which gives $N\ge Z\ge 3$. If $N=3$, then every term
is nonzero, so every $d_n=1$, and the window contains all three residues,
including $0$. That is a contradiction.
\end{proof}

\begin{proposition}\label{prop:Ashape}
If $N_A(k)=4$, the eventual cycle is $(0,1,1,2)$ and $k\equiv2\pmod4$, or it
is $(0,1,2,2)$ and $k\equiv3\pmod4$. Conversely, for each such $k$ the
corresponding $4$-periodic sequence satisfies the rule.
\end{proposition}

\begin{proof}
By the previous proof, each period of $4$ contains exactly one zero, say at
residue $r$. A block $0,a,b,c$ has increments $a,\,b-a,\,c-b,\,-c\in\{0,1\}$,
so $a=1$, $c=2$, and $b\in\{1,2\}$. The single increment equal to $0$ sits at
residue $r+2$ (if $b=1$) or $r+3$ (if $b=2$). Since $d_n=b_{n-k}$ vanishes
only when $n-k\equiv r$, we get $k\equiv2$ or $k\equiv3\pmod 4$. Checking the
converse is direct.
\end{proof}

The condition $k\equiv2,3\pmod4$ is necessary but not sufficient:
$N_A(18)=213$ and $N_A(23)=412$. There is, however, an exact test that you
can employ on a computer.\footnote[2]{The words ``gap'' and ``run'' within 
Proposition~\ref{prop:Ashape} and its proof are defined in the first Addendum.}

\begin{proposition}[Capture criterion]
Let $k\equiv2$ or $3\pmod4$ and $G=\lceil (k-1)/4\rceil$. Then $N_A(k)=4$ if
and only if, beyond index $k$, the zeros of $x$ at some point form $G$
consecutive gaps of exactly $4$.
\end{proposition}

\begin{proof}
The ``only if'' part is clear. For ``if'', suppose the run of zeros ends at
$z$. As in the previous proof, every block between consecutive zeros reads
$0,1,*,2$. So $b$ is known on $[z-4G,z]$, an interval of length
$4G+1\ge k$, and $x_z=0$. The data $(x_z;b_{z-k+1},\dots,b_z)$ determine
everything that follows, and they coincide with the data of the $4$-periodic
solution of Proposition~\ref{prop:Ashape} with a zero at $z$. So $x$ follows
that solution from $z$ on.
\end{proof}

\begin{conjecture}\label{conj:A}
$N_A(4m^2+3)=4$ for all $m\ge0$, and $N_A(3^{2j+1})=4$ for all $j\ge0$.
\end{conjecture}

The evidence is modest: the first family has been checked for $m\le 11$ (up
to $k=487$), the second only for $k=3,27,243$. Computer experiments suggest a
heuristic ``race'' picture. Long runs of zeros spaced $4$ apart appear,
copy themselves about $k$ steps later, and die out at random. Capture happens
when one run happens to reach length $G$, unless the sequence first falls
into some other long cycle. Roughly three quarters of the $k\equiv2,3\pmod4$
up to $411$ give $N_A(k)=4$. Statistics gathered before capture do not seem
to distinguish the two outcomes. This makes us suspect that no simple
sufficient condition on $k$ exists. But that is a guess, not a theorem!

\section{Summary}

\begin{center}
\renewcommand{\arraystretch}{1.25}
\begin{tabular}{|p{2.1cm}|p{3.4cm}|p{4.6cm}|p{3.6cm}|}\hline
Rule & Character & Proved & Conjectured\\\hline
(C) & linear over $\F$ & $\mu=1$; $N=\ord(x^k+x+1)$; divisibility results; Bartholdi's theorems
& $N(2^j-1)=2^{N(j)}-1$\\\hline
(B) & permutation of $(\Z/3)^k$ & $\mu=0$; exact count of states of period dividing $4k$; $N\nmid 4k$
& $N/3^k$ looks uniformly random\\\hline
(A) & random-looking & $N\ge4$; shape of $4$-cycles; $N=4\Rightarrow k\equiv2,3\ (4)$; capture criterion
& $N(4m^2+3)=4$, $N(3^{2j+1})=4$\\\hline
\end{tabular}
\end{center}

\section{Addendum: short runs of zeros in recurrence (A)}

	\subsection*{The setting}
	
	As in the main paper, fix $k\ge 2$ and let $x_1,x_2,\dots\in\Z/3$ be given by
	\[
	x_1=\dots=x_{k-1}=0,\qquad x_k=1,\qquad
	x_n = x_{n-1}+x_{n-k}^2 \quad (n\ge k+1).
	\]
	Write $b_n=\ind{x_n\neq 0}$, which is $1$ if $x_n\neq 0$ and $0$ if $x_n=0$.
	Since $c^2=\ind{c\neq 0}$ in $\Z/3$, the rule says that the \emph{step}
	\[
	d_n = x_n - x_{n-1}
	\]
	satisfies
	\begin{equation}\label{eq:step}
	d_n = b_{n-k}\in\{0,1\}\qquad\text{for every } n\ge k+1 .
	\end{equation}
	So once we are past index $k$, the sequence can only stay put or go up by
	$1$ (mod $3$) at each step. Everything below follows from this observation.
	
	In the main paper we saw that when $k\equiv 2$ or $3 \pmod 4$, the zeros of
	$x$ can line up $4$ apart forever, which gives a $4$-cycle. Here we show the
	opposite for the other two residue classes: for those $k$, zeros spaced $4$
	apart can only form \emph{short} runs.
	
	\subsection*{What we mean by a run}
	
	\begin{definition}\label{def:run}
		Let $r\ge k+1$ and $g\ge 1$ be integers. We say that $x$ has a
		\emph{run of zeros spaced exactly $4$ apart, with $g$ gaps, starting at $r$},
		if
		\begin{enumerate}
			\item $x_{r}=x_{r+4}=x_{r+8}=\dots=x_{r+4g}=0$, and
			\item $x_n\neq 0$ for every other $n$ with $r<n<r+4g$.
		\end{enumerate}
		In words: the $g+1$ positions $r,r+4,\dots,r+4g$ are zeros, there are no
		other zeros between $r$ and $r+4g$, and so the gap between each pair of
		consecutive zeros in this stretch is exactly $4$. The number $g$ counts the
		\emph{gaps}, not the zeros.
	\end{definition}
	
	For example, if the zeros of $x$ near some index $r\ge k+1$ sit at
	\[
	r,\quad r+4,\quad r+8,\quad r+11,
	\]
	then positions $r,r+4,r+8$ form a run with $g=2$ gaps. Position $r+11$ is
	not part of it, because the gap from $r+8$ to $r+11$ is $3$.
	
	We ask for $r\ge k+1$ so that \eqref{eq:step} holds for every step inside
	the run. (Every step $n$ we use satisfies $n>r\ge k+1$.) Nothing is
	assumed about what happens before $r$ or after $r+4g$, so the run need not
	be maximal.
	
	\subsection*{The result}
	
	\begin{theorem}\label{thm:main}
		Suppose $k\equiv 0$ or $k\equiv 1 \pmod 4$. Then every run of zeros spaced
		exactly $4$ apart (in the sense of Definition~\ref{def:run}) has at most
		$\lfloor (k-1)/4\rfloor$ gaps, that is,
		\[
		g\le \left\lfloor \frac{k-1}{4}\right\rfloor .
		\]
		Concretely, $g\le \frac k4-1$ if $k\equiv 0\pmod 4$, and
		$g\le \frac{k-1}{4}$ if $k\equiv 1\pmod 4$.
	\end{theorem}
	
	The proof uses one simple lemma about what a single block of a run must
	look like.
	
	\begin{lemma}\label{lem:block}
		Inside a run as in Definition~\ref{def:run}, every block of five
		consecutive terms from one zero to the next reads
		\[
		x_{m},\,x_{m+1},\,x_{m+2},\,x_{m+3},\,x_{m+4} \;=\; 0,\;1,\;*,\;2,\;0
		\qquad (m=r,r+4,\dots,r+4g-4),
		\]
		where $*$ is $1$ or $2$. Consequently, for $r<n\le r+4g$:
		\begin{itemize}
			\item if $n\equiv r\pmod 4$, then $d_n=0-2=1$ (the step that lands on a zero);
			\item if $n\equiv r+1\pmod 4$, then $d_n=1-0=1$ (the step that leaves a zero).
		\end{itemize}
	\end{lemma}
	
	\begin{proof}
		Write the block as $0,a,c,e,0$ with $a,c,e\neq 0$. Its four steps are
		$a,\;c-a,\;e-c,\;-e$, and by \eqref{eq:step} each of them is $0$ or $1$.
		\begin{itemize}
			\item The first step is $a\in\{0,1\}$, and $a\neq0$, so $a=1$.
			\item The last step is $-e\in\{0,1\}$, so $e\in\{0,2\}$. Since $e\neq0$,
			we get $e=2$.
		\end{itemize}
		The middle entry $c$ is $1$ or $2$, and either value gives steps in
		$\{0,1\}$. The two bullet points then read off the first and last steps.
	\end{proof}
	
	\begin{proof}[Proof of Theorem~\ref{thm:main}]
		Take a run with $g$ gaps starting at $r$. Suppose, aiming for a
		contradiction, that the run is long, in the sense that
		\[
		4g \ge k .
		\]
		Look at the step at position $n=r+k$. Because $1\le k\le 4g$, this position
		lies inside the run: $r<n\le r+4g$. By \eqref{eq:step},
		\[
		d_{r+k} = b_{r} = 0,
		\]
		because $x_r=0$. In words: the rule ``looks back'' exactly $k$ places, sees
		the zero at the start of the run, and therefore forces a step of $0$.
		
		Now use Lemma~\ref{lem:block} to compute the same step a second way.
		\begin{itemize}
			\item \textbf{Case $k\equiv 0\pmod 4$.} Then $n=r+k\equiv r\pmod 4$, so $n$
			is one of the zero positions $r+4,\dots,r+4g$. By the lemma, $d_n=1$.
			\item \textbf{Case $k\equiv 1\pmod 4$.} Then $n=r+k\equiv r+1\pmod 4$, so
			$n=r+4i+1$ for some $i\ge 0$. Since $n\le r+4g$, we have $4i+1\le 4g$, so
			$i\le g-1$, and $n$ is the second entry of a block. By the lemma, $d_n=1$.
		\end{itemize}
		Either way $d_n=1$, which contradicts $d_n=0$. So the run cannot be long,
		and therefore
		\[
		4g\le k-1,\qquad\text{that is,}\qquad g\le \frac{k-1}{4}.
		\]
		Since $g$ is an integer, $g\le\lfloor (k-1)/4\rfloor$. (When
		$k\equiv 0\pmod 4$, the largest multiple of $4$ that is at most $k-1$ is
		$k-4$, so $g\le k/4-1$. When $k\equiv 1\pmod 4$, we get exactly
		$g\le (k-1)/4$.)
	\end{proof}
	
	\subsection*{An example and some remarks}
	
	\begin{example}
		Take $k=4$. Then $\lfloor 3/4\rfloor=0$, so the theorem says that beyond
		index $4$ there is \emph{not even one} pair of consecutive zeros exactly $4$
		apart. You can check this against the first $60$ terms of the sequence, shown
		below with each row starting at the index $n$ in the left column; the zeros
		beyond index $4$ are in bold.
		\begin{center}
			\renewcommand{\arraystretch}{1.15}
			\[
			\begin{array}{r|cccccccccc}
			n & \multicolumn{10}{c}{x_n,\ x_{n+1},\ \dots,\ x_{n+9}}\\ \hline
			1  & 0 & 0 & 0 & 1 & 1 & 1 & 1 & 2 & \mathbf{0} & 1 \\
			11 & 2 & \mathbf{0} & \mathbf{0} & 1 & 2 & 2 & 2 & \mathbf{0} & 1 & 2 \\
			21 & \mathbf{0} & \mathbf{0} & 1 & 2 & 2 & 2 & \mathbf{0} & 1 & 2 & \mathbf{0} \\
			31 & \mathbf{0} & 1 & 2 & 2 & 2 & \mathbf{0} & 1 & 2 & \mathbf{0} & \mathbf{0} \\
			41 & 1 & 2 & 2 & 2 & \mathbf{0} & 1 & 2 & \mathbf{0} & \mathbf{0} & 1 \\
			51 & 2 & 2 & 2 & \mathbf{0} & 1 & 2 & \mathbf{0} & \mathbf{0} & 1 & 2
			\end{array}
			\]
		\end{center}		
		
\noindent Past index $4$, the zeros sit at
		\[
		9,\ 12,\ 13,\ 18,\ 21,\ 22,\ 27,\ 30,\ 31,\ 36,\ 39,\ 40,\ 45,\ 48,\ 49,\
		54,\ 57,\ 58,
		\]
		and the gaps between consecutive ones are only ever $1$, $3$ or $5$, never
		$4$. For $k=5$ the bound is $\lfloor 4/4\rfloor=1$: single gaps of $4$ are
		allowed, but never two in a row.
	\end{example}
	
	\begin{remark}[Why $k\equiv 2,3$ are different]
		The same argument with $k\equiv 2$ or $3\pmod 4$ lands on a step
		$n\equiv r+2$ or $r+3\pmod 4$. Those are the steps \emph{into} and
		\emph{out of} the middle entry $*$, and one of them may well be $0$. So
		there is no contradiction, and indeed the $4$-cycles $(0,1,1,2)$ and
		$(0,1,2,2)$ of the main paper are runs that go on forever.
	\end{remark}
	
	\begin{remark}[The bound is sharp in practice]
		Empirical data suggest that Theorem~\ref{thm:main} is the best possible. For
		instance, runs with exactly $\lfloor (k-1)/4\rfloor$ gaps were observed for
		$k=245,297,333$ (where $k\equiv1$) and for $k=288,320,340$ (where
		$k\equiv0$). Proving that such maximal runs always occur would be a nice
		exercise in experimentation, but maybe a hard one in theory!
	\end{remark}

\section{Addendum:  The moment of capture in recurrence (A)}

\section*{The setting}

As in the main paper, fix $k\ge 2$ and let $x_1,x_2,\dots\in\Z/3$ satisfy
\[
x_1=\dots=x_{k-1}=0,\qquad x_k=1,\qquad x_n=x_{n-1}+x_{n-k}^2\quad(n\ge k+1).
\]
Write $b_n=\ind{x_n\neq0}$. Throughout, suppose that $k\equiv 2$ or
$3\pmod 4$ and that the eventual period is $N_A(k)=4$. Put
\[
G=\left\lceil\frac{k-1}{4}\right\rceil .
\]
We use the same notion of run as in the first Addendum: zeros at
$r,r+4,\dots,r+4g$ with $r\ge k+1$ and no other zeros in between. Such a
run has $g$ \emph{gaps}.

By the capture criterion of the main paper, the sequence settles into its
$4$-cycle as soon as some run reaches $G$ gaps. This suggests the following
definitions.

\begin{definition}
	The \emph{capture point} $z=z(k)$ is the smallest index $z$ such that $x$ has
	a run of zeros spaced exactly $4$ apart, with $G$ gaps, ending at $z$. Its
	first zero is $s=z-4G$.
\end{definition}

\begin{definition}
	As usual, the \emph{pre-period} $\mu=\mu_A(k)$ is the least $\mu\ge0$ with
	$x_{n+4}=x_n$ for all $n>\mu$. Equivalently, $\mu+4$ is the last index $n$ at
	which $x_n\neq x_{n-4}$.
\end{definition}

\section*{The guess}

\begin{conjecture}\label{conj}
	For every $k\equiv 2,3\pmod 4$ with $N_A(k)=4$,
	\[
	z(k)-\mu_A(k)=6 .
	\]
	Equivalently, the last disagreement with the $4$-cycle happens exactly two
	steps before the capture point:
	\[
	x_{z-2}\neq x_{z-6},\qquad\text{and}\qquad x_n=x_{n-4}\ \text{ for all } n>z-2 .
	\]
\end{conjecture}

In words: capture always ``clicks in'' in the same way. The final run of
length $G$ is completed, and everything after position $z-2$ is already in
step with the cycle. The cycle is never reached earlier than that, and
there is never a stray mismatch after it.

\begin{remark}[What it says about the run's neighborhood]
	Inside the capturing run, every block between consecutive zeros reads
	$0,1,*,2$. Its middle entry $*$ is determined by a value of $b$ exactly $k$
	places back. Following this through, the conjecture amounts to a small
	condition on the terms just before the run starts at $s$:
	\[
	\begin{array}{ll}
	k\equiv 2\pmod 4: & b_{s-4}=1,\\[2pt]
	k\equiv 3\pmod 4: & b_{s-5}=0\ \text{ and }\ b_{s-1}=1.
	\end{array}
	\]
	So the conjecture is really a statement about how the capturing run begins.
\end{remark}

\section*{The evidence}

The identity $z-\mu=6$ was observed, with no exceptions, for
\[
k=3,\ 6,\ 7,\ 10,\ 11,\ 14,\ 15,\ 19,\ 199,\ 207,\ 243,\ 294,\ 303,\ 327,\ 403,
\]
which covers both residue classes ($k=6,10,14,294$ are $\equiv2$; the rest
are $\equiv3\pmod4$). The pre-periods range from $\mu=5$ (for $k=3$) to
$\mu=104{,}321{,}125{,}322$ (for $k=403$), where the capture point is
$z=104{,}321{,}125{,}328$. A nice first project would be to test more values
of $k$, then to look for the reason. Why should the sequence, after
wandering for $10^{11}$ steps, always arrive at its cycle through the same
narrow door?

\section*{Acknowledgements}
I gratefully acknowledge Bob Harder's original Fortran-based computations in 2000 
(now extended by my C program).
Conversations with both Anthropic Claude Opus 5.5 and OpenAI GPT-6 Luna 
were essential in the preparation of this paper.  No AI funding whatsoever 
for this study was received. Proposition~\ref{prop:Ashape} was anticipated 
in \cite{MP}. My interest in the subject arose long ago in \cite{CF}; 
$x_n = 2\bigl(x_{n-2}^2+x_{n-2}+x_{n-4k-3}^2+x_{n-4k-3}+x_{n-8k-4}^2+x_{n-8k-4}\bigr) \pmod 3$ 
is understood, but our seemingly simpler recurrence (A) here defies explanation.

\end{document}